\documentclass[11pt, a4paper]{article}
\usepackage[english]{babel}
\usepackage{fancyhdr}
\usepackage{latexsym, amssymb, amsmath, amscd, amscd, amsthm, amsxtra}
\usepackage [latin1]{inputenc}
 \usepackage{graphicx}
\usepackage[all]{xy}
\usepackage{placeins}
\usepackage{setspace}

\usepackage{fancyhdr}
\usepackage{fancybox}
\usepackage{latexsym, amssymb, amsmath, amscd, amscd, amsthm, amsxtra}
\usepackage{subfigure}
\usepackage{graphicx}
\usepackage[all]{xy}

\usepackage{amsthm}
 \usepackage{indentfirst}
\theoremstyle{plain}
\newtheorem{prop}{Proposition}[section]
\newtheorem{lem}[prop]{Lemma}

\newtheorem{cor}[prop]{Corollary}
\newtheorem{thm}[prop]{Theorem}

\newtheorem*{prop*}{Proposition}
\newtheorem*{lem*}{Lemma}
\newtheorem*{sublem*}{Sublemma}
\newtheorem*{cor*}{Corollaire}
\newtheorem*{thm*}{Theorem}
\newtheorem*{hypo*}{Hypothesis}
\newtheorem*{question*}{Question}
\newtheorem*{conjecture*}{Conjecture}
\newtheorem*{scholum*}{Scholum}
\newtheorem{defn}[prop]{Definition}
\newtheorem*{defn*}{Definition}

\theoremstyle{slanted}

\newtheorem*{ex*}{Example}
\newtheorem*{exs*}{Examples}

\newtheorem*{rmk*}{Remark}
\newtheorem*{rmks*}{Remarks}

\newtheorem*{notation*}{Notation}

\theoremstyle{definition}

\newtheorem*{con*}{Construction}

\theoremstyle{remark}

\newtheorem*{warning*}{Warning}
\newtheorem*{shortnote*}{Note}
\newtheorem*{claim*}{Claim}
\newtheorem*{axiom*}{Axiom}

\begin{document}

\title{Syzygies of secant varieties of curves of genus 2}

\author{Li Li}

\date{}

\newcommand{\Addresses}{{
  \bigskip
  \footnotesize

  \textsc{Humboldt-Universit\"{a}t zu Berlin, Institut f\"{u}r Mathematik, Unter den Linden 6, 10099 Berlin, Germany}\par\nopagebreak
  \textit{E-mail address}, \texttt{lili@hu-berlin.de}

}}

\maketitle

\begin{abstract}
Ein, Niu and Park showed in \cite{ENP} that if the degree of the line bundle $L$ on a curve of genus $g$ is at least $2g+2k+1$, the $k$-th secant variety of the curve via the embedding defined by the complete linear system of $L$ is normal, projectively normal and arithmetically Cohen-Macaulay, and they also proved some vanishing of the Betti diagrams. However, the length of the linear strand of weight $k+1$ of the resolution of the secant variety $\Sigma_k$ of a curve of $g\geq2$ is still mysterious. In this paper we calculate the complete Betti diagrams of the secant varieties of curves of genus $2$ using Boij-S\"{o}derberg theory. The main idea is to find the pure diagrams that contribute to the Betti diagram of the secant variety via calculating some special positions of the Betti diagram.
\end{abstract}
\section{\label{S3}Introduction}
Let $C$ be a smooth projective curve of genus $g$ over the complex field $\mathbb{C}$. Let $L$ be a very ample line bundle on $C$ such that its complete linear system defines an embedding $C\hookrightarrow\mathbb{P}^r$. For an integer $k\geq0$, define the $k$-th secant variety $\Sigma_k=\Sigma_k(C,L)\subset\mathbb{P}^r$ to be the Zariski closure of the union of all $(k+1)$-secant $k$-planes to $C$ in $\mathbb{P}^r$. Formally, the secant varieties can be realized as images of the projectivizations of secant bundles. Concretely, let $pr:C\times C_k\to C$ be the projection on the first factor, where $C_k$ denotes the $k$-th symmetric product of $C$. We have the canonical morphism $\sigma:C\times C_k\to C_{k+1}$ which sends $(x,\xi)$ to $x+\xi$. Then the $(k+1)$-th secant sheaf $E_{k+1,L}$ or simply $E_L$ is defined to be the rank $k+1$ vector bundle $\sigma_*pr^*L$. The projective bundle is defined to be the projectivization $B^k(L):=\mathbb{P}(E_L)$. Let $\mathcal{O}_{B^k}(1):=\mathcal{O}_{\mathbb{P}(E_L)}(1)$ be the tautological bundle on $B^k$. It induces a morphism $\beta:B^k\to\mathbb{P}H^0(C_{k+1},E_L)=\mathbb{P}^r$. The $k$-th secant variety $\Sigma_k$ can be realized as the image $\beta$. We may assume that $r\geq 2k+3$, because by \cite{LA} the secant varieties of curves always have the expected dimension, meaning that for $\Sigma_k\subset\mathbb{P}^r$, we have $\dim\Sigma_k=\min\{2k+1,r\}$. Concretely, $\Sigma_k=\mathbb{P}^r$ if $r\leq 2k+1$ and $\Sigma_k$ is a hypersurface with degree calculated in Proposition 5.10 of \cite{ENP} if $r=2k+2$. Let $\mathcal{O}_{\Sigma_k}(1)$ be the line bundle on $\Sigma_k$ that is the pullback of $\mathcal{O}_{\mathbb{P}^r}(1)$.

We are interested in the syzygies of $\Sigma_k$. To that end, we introduce some notations. Let $X$ be a projective variety embedded in $\mathbb{P}^r=\mathbb{P}H^0(X,L)$ via a very ample line bundle $L$. Let $S=\mathrm{Sym}H^0(X,L)$ and $R=\bigoplus\limits_{m\geq0}H^0(X,B\otimes mL)$ viewed as a graded $S$-module, where $B$ is a coherent sheaf on $X$. By Hilbert's syzygy theorem, there is a minimal graded free resolution of $R$ over $S$:
$$\cdots\longrightarrow E_2\longrightarrow E_1\longrightarrow E_0\longrightarrow R\longrightarrow0,$$
where $E_p$ is a free $S$-module. We define the Koszul cohomology group $K_{p,q}(X,B;L)$ to be the suitable $\mathbb{C}$-linear space such that $E_p=\bigoplus\limits_qK_{p,q}(X,B;L)\otimes_\mathbb{C}S(-p-q)$.

The Koszul cohomology group $K_{p,q}(X,B;L)$ can be identified with the homology of the following differentials:
$$\bigwedge\limits^{p+1}H^0(L)\otimes H^0(B\otimes(q-1)L)\to\bigwedge\limits^{p}H^0(L)\otimes H^0(B\otimes qL)\to\bigwedge\limits^{p-1}H^0(L)\otimes H^0(B\otimes(q+1)L),$$
 where the morphism $\bigwedge\limits^{p}H^0(L)\otimes H^0(B\otimes qL)\to\bigwedge\limits^{p-1}H^0(L)\otimes H^0(B\otimes(q+1)L)$ is given by $$x_1\wedge\cdots\wedge x_p\otimes M\mapsto\sum\limits_{i=1}^p(-1)^ix_1\wedge\cdots\wedge\widehat{x_i}\wedge\cdots\wedge x_p\otimes x_iM,$$
 for arbitrary $p,q\in\mathbb{Z}$.
 If $B=\mathcal{O}_X$ the structure sheaf on $X$, the Koszul cohomology groups are simply denoted as $K_{p,q}(X,L)$. We have the Betti diagram of which the column $p$ and row $q$ is $\dim K_{p,q}(X,L)=:b_{p,q}$.\\

 \begin{tabular}{|c| c| c| c| c| }
\hline
 &$0$ & $1$ & $2$ & $\cdots$\\
 \hline
 $0$& $b_{0,0}$ &$b_{1,0}$ &$b_{2,0}$ &$\cdots$\\
 \hline
 $1$& $b_{0,1}$ &$b_{1,1}$ &$b_{2,1}$ &$\cdots$\\
 \hline
 $2$& $b_{0,2}$ &$b_{1,2}$ &$b_{2,2}$ &$\cdots$\\
 \hline
 $\vdots$& $\vdots$  & $\vdots$  & $\vdots$    &  $\cdots$   \\
 \hline
\end{tabular}\\

In this paper, we will focus on curves of genus $2$ when the kernel bundle of the evaluation map of the canonical bundle on the curve is a line bundle, and thus easier to describe. We assume that $r\geq 2k+3$, which means that $\deg(L)\geq2k+5$. It was shown from \cite{ENP} that $\Sigma_k$ satisfies the property $N_{k+2,r-2k-3}$ i.e. $\dim K_{i,j}(\Sigma_k,\mathcal{O}_{\Sigma_k}(1))=0$ for $i\leq r-2k-3$ and $j\geq k+2$, and that $\dim K_{r-2k-1,2k+2}(\Sigma_k,\mathcal{O}_{\Sigma_k}(1))=k+2$. Furthermore, by Danila's theorem(see \cite{GD}), we have $\dim K_{i,j}(\Sigma_k,\mathcal{O}_{\Sigma_k}(1))=0$ for $j\leq k$, except from $(i,j)=(0,0)$, for which we have $\dim K_{0,0}(\Sigma_k,\mathcal{O}_{\Sigma_k}(1))=1$. We will prove the followings:
\begin{thm}For the the secant variety $\Sigma_k$ of a genus $2$ curve, we have\\ $\dim K_{r-2k-2,2k+2}(\Sigma_k,\mathcal{O}_{\Sigma_k}(1))=r-2k-1$ and $K_{r-2k-1,2k+1}(\Sigma_k,\mathcal{O}_{\Sigma_k}(1))=0$.
\end{thm}
Based on the values calculated above, we can determine the length of the linear strand together with the unknown Koszul cohomology group $K_{r-2k-1,2k+1}$.
\begin{thm}Let $\Sigma_k$ be the the secant variety  of a genus $2$ curve.\\
(1)We have $\dim K_{r-2k-2,2k+1}(\Sigma_k,\mathcal{O}_{\Sigma_k}(1))=r-k-2$.\\
(2)For $k\geq 1$, we have the equivalence $\dim K_{i,k+1}(\Sigma_k,\mathcal{O}_{\Sigma_k}(1))\not=0\iff 1\leq i\leq r-2k-3$, while for $k=0$, we have $\dim K_{i,1}(C,\mathcal{O}_C(1))\not=0\iff 1\leq i\leq r-2$.\\
(3)For $k\geq1$ and $1\leq i\leq r-2k-3$, we have
\begin{eqnarray*}
   &&\dim K_{i,k+1}(\Sigma_k,\mathcal{O}_{\Sigma_k}(1)) \\
   &=&\frac{(r-k-2)![r^3-(i+k+1)r^2-(i+k+2)r+2(k+1)(i+k+1)]}{(k+1)!(i+k+1)(i-1)!(r-i-2k-3)!(r-i-k-2)(r-i-k-1)(r-i-k)}.
\end{eqnarray*}
(4)For $k=0$ and $1\leq i\leq r-2$, we have $$\dim K_{i,1}(C,\mathcal{O}_C(1))=\frac{(r-1)!(r^2-ir-2i-2)}{(i+1)(i-1)!(r-i)!}.$$
\end{thm}
Summarizing the results above, we deduce that the Betti diagram of $\Sigma_k$($k\geq1$) is of the following shape:\\

\begin{tabular}{|c| c| c| c| c| c| c| c|}
\hline
 &$0$ & $1$ & $2$ & $\cdots$ & $r-2k-3$ & $r-2k-2$ & $r-2k-1$\\
 \hline
 $0$& $1$ & & & & & &\\
 \hline
 $\vdots$&   &   &     &   &        &         &    \\
 \hline
 $k+1$& & * & * &  $\cdots$ & *  & &\\
 \hline
 $\vdots$&   &   &     &   &        &         &    \\
 \hline
  $2k+1$&   &   &     &   &        &    $r-k-2$     &    \\
  \hline
   $2k+2$&   &   &     &   &        &   $r-2k-1$      &$k+2$    \\
 \hline
\end{tabular}\\

where the blanks are $0$ and the asterisks stand for non-zero terms, the values of which were shown in Theorem 1.2 (3).
\section{\label{S3}The vanishing of $K_{r-2k-1,2k+1}$}
Before proving the main theorem in this section, it is useful to introduce the duality property of Koszul cohomology.
\begin{prop}Let $L$ be a globally generated line bundle on a smooth projective scheme $X$ of dimension $n$ and $B$ an arbitrary line bundle on $X$. If $H^i(X,B\otimes L^{q-i})=H^i(X,B\otimes L^{q-i+1})=0$ for $i=1,\cdots,n-1$, then we have the isomorphism ${K_{p,q}(X,B;L)^\vee\cong K_{r-n-p,n+1-q}(X,\omega_X\otimes B^\vee;L)}$, where $r=h^0(X,L)-1$ and $\omega_X$ is the canonical sheaf on $X$.
\end{prop}
\begin{proof}The proof is essentially the same as that of Theorem 2.25 of \cite{AN}. However, the smoothness assumed there is not necessary because the ingredient is the Serre's duality.
\end{proof}
From now on, we assume that $C$ is a smooth projective curve of genus $2$ over $\mathbb{C}$ and $L$ is a line bundle of degree at least $2k+5$. The complete system of $L$ defines an embedding $C\hookrightarrow\mathbb{P}^r$ with $r\geq2k+3$. Let $\Sigma_k\subset\mathbb{P}^r$ be the $k$-th secant variety of $C$. Since all intermediate cohomology $H^i(\Sigma_k,\mathcal{O}_{\Sigma_k}(l))$ with $l\in\mathbb{Z}$ and $1\leq i\leq 2k$ vanish by Theorem 5.2 and Theorem 5.8 of \cite{ENP}, the duality property holds. In other words, we have ${K_{r-2k-1,2k+1}(\Sigma_k,\mathcal{O}_{\Sigma_k}(1))\cong K_{0,1}(\Sigma_k,\omega_{\Sigma_k};\mathcal{O}_{\Sigma_k}(1))^\vee}$, where $\omega_{\Sigma_k}$ is the canonical sheaf on $\Sigma_k$. Therefore to show the vanishing of this group, it is equivalent to show the following:
\begin{thm}The morphism
$$\phi:H^0(\Sigma_k,\mathcal{O}_{\Sigma_k}(1))\otimes H^0(\Sigma_k,\omega_{\Sigma_k})\to H^0(\Sigma_k,\omega_{\Sigma_k}(1))$$
is surjective with the assumption $r\geq 2k+3$.
\end{thm}
\begin{proof}By the construction of the secant varieties and Theorem 5.8 of \cite{ENP}, we know that $H^0(\Sigma_k,\mathcal{O}_{\Sigma_k}(1))=H^0(C,L)$, $H^0(\Sigma_k,\omega_{\Sigma_k})=S^{k+1}H^0(C,K)$ and\\ ${H^0(\Sigma_k,\omega_{\Sigma_k}(1))=H^0(C,K+L)\otimes S^kH^0(C,K)}$, where $K$ is the canonical divisor on $C$ and the morphism $\phi$ becomes $$H^0(C,L)\otimes S^{k+1}H^0(C,K)\to H^0(C,K+L)\otimes S^kH^0(C,K),$$ given by $$x\otimes v_1\cdots v_{k+1}\mapsto\sum\limits_{i=1}^{k+1}v_ix\otimes v_1\cdots\widehat{v_i}\cdots v_{k+1}.$$
We still denote it by $\phi$ and assume $J=\mathrm{Im}(\phi)$. Write $H^0(C,K)=\mathrm{span}\{1,v\}$, where $1$ and $v$ are seen as rational functions on an affine open subset of $C$.

From now on I omit the symbol $C$ for short. As long as the multiplication morphism ${\psi_0:H^0(K)\otimes H^0(L)\to H^0(K+L)}$ is surjective, the elements in $H^0(K+L)\otimes S^kH^0(K)$ are spanned by elements of two types, $x\otimes v^m$ and $vx\otimes v^m$, where $0\leq m\leq k$ and $x\in H^0(L)$. Clearly the element $vx\otimes v^k$ is in $J$ since $\phi(x\otimes v^{k+1})=(k+1)vx\otimes v^k$ and $x\otimes 1$ is also in $J$ since $\phi(x\otimes1)=(k+1)x\otimes 1$. We only need to show $x\otimes v^m\in J$ for $0\leq m\leq k$ and $x\in H^0(L)$. Once we have shown this, we would have $vx\otimes v^{m-1}\in J$ for all $x\in H^0(L),1\leq m\leq k$, because $\phi(x\otimes v^m)$ is a combination of $x\otimes v^m$ and $vx\otimes v^{m-1}$ for all $1\leq m\leq k$. We prove it by induction on $m$.

For $x\otimes 1$, it was discussed above.

Assume that $x\otimes v^{m-1}\in J$ for all $x\in H^0(L)$ and $1\leq m\leq k$. We want to show $x\otimes v^{m}\in J$. If we further assume that $\psi_1:H^0(K)\otimes H^0(L-K)\to H^0(L)$ is surjective, we can write $x=vx_1+y_1$ for some $x_1,y_1\in H^0(L-K)$. Note that $y_1\otimes v^m$ is in $J$ because $\phi(y_1\otimes v^m)$ is a combination of $vy_1\otimes v^{m-1}$ and $y_1\otimes v^m$, and we have $vy_1\otimes v^{m-1}\in J$ by the induction hypothesis. Therefore we are left to show $vx_1\otimes v^m\in J$ for $x_1\in H^0(L-K)$.

Assume that we have reduced the problem to show $v^px_p\otimes v^m\in J$ for all $x_p\in H^0(L-pK)$. If we further assume that $\psi_{p+1}:H^0(K)\otimes H^0(L-(p+1)K)\to H^0(L-pK)$ is surjective, then we can write $v^px_p=v^p(vx_{p+1}+y_{p+1})$ for some $x_{p+1},y_{p+1}\in H^0(L-(p+1)K)$. Note that $v^py_{p+1}\otimes v^m\in J$ since $\phi(v^py_{p+1}\otimes v^m)$ is a combination of $v^py_{p+1}\otimes v^m$ and $v^{p+1}y_{p+1}\otimes v^{m-1}$ and the latter is in $J$ by the induction hypothesis. We are left to show $v^{p+1}x_{p+1}\otimes v^m\in J$. We do the step above for $k$ times. Then we are left to show $v^kx_k\otimes v^m\in J$ for $x_k\in H^0(L-kK)$.

To show $v^kx_k\otimes v^m\in J$, it suffices to show $v^{k-1}v_k\otimes v^{m+1}\in J$ since $\phi(v^{k-1}v_k\otimes v^{m+1})$ is a combination of $v^kx_k\otimes v^m$ and $v^{k-1}v_k\otimes v^{m+1}$. Similarly, to show $v^{k-1}v_k\otimes v^{m+1}\in J$, it suffices to show $v^{k-2}v_k\otimes v^{m+2}\in J$. Finally we are left to show $v^{m}x_k\otimes v^{k}\in J$. This is clear since $\phi(v^{m-1}x_k\otimes v^{k+1})=(k+1)v^{m}x_k\otimes v^{k}$.

For the deduction above, we need the surjectivity of the morphisms
$$\psi_{p+1}:H^0(K)\otimes H^0(L-(p+1)K)\to H^0(L-pK)$$
for $-1\leq p\leq k-1$. Note that the kernel of the morphism $H^0(K)\otimes\mathcal{O}_C\to K$ is a line bundle and hence it is $-K$. Consider the exact sequence
$$0\to L\otimes K^{-(p+2)}\to H^0(K)\otimes L\otimes K^{-(p+1)}\to L\otimes K^{-p}\to0.$$
 To verify the desired surjectivity, it suffices to check $H^1(L-(p+2)K)=0$ i.e. $H^0((p+3)K-L)=0$. Note that $\deg((p+3)K-L)\leq 2(k+2)-\deg(L)=2k+2-r<0$. Therefore its global section group vanishes.
\end{proof}
\begin{cor}We have $\dim K_{r-2k-2,2k+2}(\Sigma_k,\mathcal{O}_{\Sigma_k}(1))=r-2k-1$.
\end{cor}
\begin{proof}From the duality property we have $K_{r-2k-2,2k+2}(\Sigma_k,\mathcal{O}_{\Sigma_k}(1))^\vee=K_{1,0}(\Sigma_k,\omega_{\Sigma_k};\mathcal{O}_{\Sigma_k}(1))$, where the latter group is the homology of the sequence
$$\bigwedge\limits^2H^0(\Sigma_k,\mathcal{O}_{\Sigma_k}(1))\otimes H^0(\Sigma_k,\omega_{\Sigma_k}(-1))\to H^0(\Sigma_k,\mathcal{O}_{\Sigma_k}(1))\otimes H^0(\Sigma_k,\omega_{\Sigma_k})\to H^0(\Sigma_k,\omega_{\Sigma_k}(1)).$$
We have $H^0(\omega_{\Sigma_k}(-1))=H^{2k+1}(\mathcal{O}_{\Sigma_k}(1))^\vee=0$ by Theorem 5.2 of \cite{ENP}. By Theorem 2.2, the map on the right is surjective, Therefore

$\begin{array}{lcl}
\dim K_{r-2k-2,2k+2}&=&\dim\ker(H^0(\mathcal{O}_{\Sigma_k}(1))\otimes H^0(\omega_{\Sigma_k})\to H^0(\omega_{\Sigma_k}(1)))\\
&=&\dim H^0(\mathcal{O}_{\Sigma_k}(1))\otimes H^0(\omega_{\Sigma_k})-\dim H^0(\omega_{\Sigma_k}(1))\\
&=&\dim H^0(L)\otimes S^{k+1}H^0(K)-\dim H^0(K+L)\otimes S^kH^0(K)\\
&=&(r+1)(k+2)-(r+3)(k+1)\\
&=&r-2k-1.
\end{array}$\\
\end{proof}
\section{\label{S3}Boij-S\"{o}derberg theory}
In this section we recall some useful facts from Boij-S\"{o}derberg theory. One may refer to \cite{ES09} and \cite{TL} for this part.
\begin{defn}The rational vector space $\mathbb{B}=\bigoplus\limits_{-\infty}^\infty\mathbb{Q}^{n+1}$ is called the space of rational Betti diagrams with $n+1$ columns and rows numbered by integers.
\end{defn}
\begin{defn}A pure diagram $\beta(\mathbf{e})$ is the diagram characterised by a degree sequence \\ ${\mathbf{e}=(e_0,\cdots,e_n)\in\mathbb{Z}^{n+1}}$ with $e_0<e_1<\cdots<e_n$, for which the Betti numbers are defined by
$$
\dim K_{p,q}(\beta(\mathbf{e}))=\left\{
\begin{aligned}
&n!\prod\limits_{i\not=e_p}\frac{1}{|e_i-e_p|},\quad p+q=e_p\\
&0,\quad else
\end{aligned}
\right.
$$
\end{defn}
From Theorem 0.1 and Theorem 0.2 of \cite{ES09}, we can get the theorem of decompositions of the Betti diagram of a Cohen-Macaulay module into combinations of pure diagrams.
\begin{thm}Let $S$ be the polynomial ring $\mathbb{C}[x_0,\cdots,x_r]$. For any finitely generated Cohen-Macaulay graded $S$-module $M$, its Betti diagram $\beta(M)$ can be (not necessarily uniquely) decomposed as $\sum\limits_{\mathbf{e}}c_\mathbf{e}\beta(\mathbf{e})$ with $c_\mathbf{e}\geq0$.
\end{thm}
Next we introduce the multiplicity of a module and extend it to the formal diagrams.
\begin{defn}Let $S$ be as above and $M$ a graded $S$-module. The Hilbert series\\ $HS_M(t)=\sum\limits_j\dim M_jt^j$ of $M$ can be uniquely represented as$$HS_M(t)=\frac{HN_M(t)}{(1-t)^{\dim M}}.$$
Define the multiplicity of $M$ to be $HN_M(1)$.
\end{defn}
From Corollary of \cite{HM84} and Theorem 1.2 of \cite{HM85}, we know that
\begin{thm}For a pure resolution
    $$0\to S(-d_n)^{b_n}\to\cdots\to S(-d_1)^{b_1}\to S\to R\to0,$$
having conditions $b_i=|\prod\limits_{j\not=i,j\geq1}\frac{d_i}{d_j-d_i}|$ for $1\leq i\leq n$, where $R$ is a graded $S$-module and in fact Cohen-Macaulay with the numerical conditions of $b_i$, the multiplicity of $R$ is $\mu(R)=\displaystyle\frac{\prod\limits_{i=1}^n d_i}{n!}$.
\end{thm}
Recall that the Hilbert functions of a module are totally determined linearly by the Betti numbers(Corollary 1.10 of \cite{ES05}). Therefore they can be generalized to formal diagrams. Furthermore, the multiplicity of a formal diagram can be defined. The multiplicity should be also totally determined linearly by the Betti numbers. So we can calculate the multiplicities of the pure diagrams defined in Definition 3.2.
\begin{prop}The multiplicities of the pure diagrams defined in Definition 3.2 are $1$.
\end{prop}
\begin{proof}We first check that when normalized such that $\dim K_{0,0}=1$, the Betti numbers of the pure diagrams satisfy the conditions in Theorem 3.5. In fact, we have $$\frac{\dim K_{i,e_i-i}}{\dim K_{0,0}}=\frac{\prod\limits_{j\not=i,j\geq0}\frac{1}{|e_j-e_i|}}{\prod\limits_{j\not=0}\frac{1}{|e_j|}}=|\prod\limits_{j\not=i,j\geq1}\frac{e_i}{e_j-e_i}|.$$
Since the multiplicity is totally determined by the Betti numbers, and the Betti numbers of the pure diagrams divided by $\dim K_{0,0}$ satisfy the conditions in Theorem 3.5, the result of Theorem 3.5 is formally generalized. Then the multiplicity of $\displaystyle\frac{\beta(\mathbf{e})}{\dim K_{0,0}}$ is $\displaystyle\frac{\prod\limits_{i=1}^ne_i}{n!}$. As a result, the multiplicity of $\beta(\mathbf{e})$ is ${\dim K_{0,0}}\cdot\displaystyle\frac{\prod\limits_{i=1}^ne_i}{n!}=1$.
\end{proof}
\section{\label{S3}The complete Betti diagram of $\Sigma_k$}
We first show that the multiplicity of a projective variety coincides with its degree.
\begin{prop}Let $V\subset\mathbb{P}^r$ be a projective variety. Let $S=\mathbb{C}[x_0,\cdots,x_r]$ and $I$ be the vanishing ideal of $V$. Then the multiplicity of $S/I$ coincides with the degree of $V$. Here the degree means the number of points in the intersections $V\cap H_1\cap\cdots\cap H_{\dim V}$, where $H_i$'s are general hyperplanes in $\mathbb{P}^r$.
\end{prop}
\begin{proof}The Hilbert series $HS_V(t)$ of $V$ is of the form  $HS_V(t)=\sum\limits_jH_V(j)t^j$, where $H_V(t)$ is the Hilbert polynomial of $V$. Then $tHS_V(t)=\sum\limits_jH_V(j)t^{j+1}$. Taking the difference of these two equations, we get $(1-t)HS_V(t)\cong\sum\limits_j[H_V(j)-H_V(j-1)]t^j=\sum\limits_jH_{V\cap H_1}(j)t^j$. Here the symbol $\cong$ means that they are equal from a term of sufficiently large degree. Inductively, we get $(1-t)^{\dim V}HS_V(t)=\sum\limits_{j=0}^Nc_jt^j+\sum\limits_{j=N+1}^\infty Dt^j$, where $D=\deg(V)$. Multiplied by $(1-t)$ on both sides, it is deduced that $$HS_V(t)=\frac{(1-t)P(t)+Dt^{N+1}}{(1-t)^{\dim V+1}}=\frac{(1-t)P(t)+Dt^{N+1}}{(1-t)^{\dim S/I}}.$$
By the definition of the multiplicity, we see that it is equal to the degree.
\end{proof}
We fix some notations. Observe that for the Betti diagram of $\Sigma_k$, $\dim K_{1,k+1}\not=0$ and the lower right corner is $(r-2k-1,2k+2)$. These mean that the sequence $\mathbf{e}$ that characterizes the diagram $\beta(\mathbf{e})$ composing the Betti diagram of $\Sigma_k$ has the first non-zero entry starting from $k+2$, and the last entry at most $r+1$.

Let $\mathbf{i}=(i_0,\cdots,i_k)$ with $0\leq i_0\leq i_1\leq\cdots\leq i_k$. Let $\pi_k(\mathbf{i};d)$ be the pure diagram associated to the sequence
$$\{0,1,\cdots,r+1\}\setminus\{1,\cdots,k+1,r+1-(i_k+k),\cdots,r+1-i_0\},$$
where $r=d-2$.

From Theorem 1.3 of \cite{SV09} we know that the linear strand of weight $k+1$ of $\Sigma_k$ has length at least $r-2k-3$. This implies that $k+1+r-2k-3<r+1-(i_k+k)\Rightarrow i_k\leq2$.

Write $\overline{\beta}(\Sigma_k)=\beta(\Sigma_k)/\deg(\Sigma_k)$. Then the diagram $\overline{\beta}(\Sigma_k)$ has multiplicity $1$. It was shown in \cite{TL} that
\begin{lem}There is a decomposition of the diagram $\overline{\beta}(\Sigma_k)=\sum\limits_{\mathbf{i}|i_k\leq2}c_{\mathbf{i};d}\pi_k(\mathbf{i};d)$ with ${\sum\limits_{\mathbf{i}|i_k\leq2}c_{\mathbf{i};d}=1}$ and $0\leq c_{\mathbf{i};d}\leq1$.
\end{lem}
Now we can prove our main theorem:
\begin{thm}Assume that $k\geq1$. Then for the secant variety $\Sigma_k$ of a genus $2$ curve, we have:\\
(1)$\dim K_{r-2k-2,2k+1}(\Sigma_k,\mathcal{O}_{\Sigma_k}(1))=r-k-2$;\\
(2)$\dim K_{i,k+1}(\Sigma_k,\mathcal{O}_{\Sigma_k}(1))\not=0\iff 1\leq i\leq r-2k-3$.
\end{thm}
\begin{proof}Consider $K_{r-2k-1,2k+2}$. For a pure diagram characterised by $\mathbf{e}$ contributing to $K_{r-2k-1,2k+2}$, we have $e_{r-2k-1}=r-2k-1+2k+2=r+1$. This implies that only those $\mathbf{i}$ with $i_0>0$ contribute to $K_{r-2k-1,2k+2}$. Applying Lemma 4.2 to $K_{r-2k-1,2k+2}$, we have
\begin{equation}
  (k+2)/\deg(\Sigma_k)=\sum\limits_{i_0>0}c_{\mathbf{i};d}\dim K_{r-2k-1,2k+2}(\pi_k(\mathbf{i};d)).\label{X4}
\end{equation}
By the definition, $\dim K_{r-2k-1,2k+2}(\pi_k(\mathbf{i};d))=\displaystyle\frac{(r-2k-1)!}{(r+1)\prod\limits_{j\not=r-2k-1}|r+1-e_j|}$. \\
Since we have
\begin{eqnarray*}
\prod\limits_{j\not=r-2k-1}|r+1-e_j|\cdot r\cdots(r-k)\prod\limits_{j=0}^k(i_j+j)&=&r!\\
\prod\limits_{j\not=r-2k-1}|r+1-e_j|&=&\frac{(r-k-1)!}{\prod\limits_{j=0}^k(i_j+j)},
\end{eqnarray*}
it is deduced that $\dim K_{r-2k-1,2k+2}(\pi_k(\mathbf{i};d))=\displaystyle\frac{\prod\limits_{j=0}^k(i_j+j)}{(r+1)\prod\limits_{j=r-2k}^{r-k-1}j}$.\\
It can be deduced from Proposition 5.10 of \cite{ENP} that, when $k\geq1$, we have
\begin{eqnarray*}
\deg(\Sigma_k)&=&\binom{r-k}{k+1}+2\binom{r-k-1}{k}+\binom{r-k-2}{k-1}\\
&=&\frac{(r^2+r-2k-2)\prod\limits_{j=r-2k}^{r-k-2}j}{(k+1)!}.
\end{eqnarray*}
Therefore the equation \eqref{X4} becomes
\begin{eqnarray*}
k+2&=&\sum\limits_{i_0>0}\frac{(r^2+r-2k-2)\prod\limits_{j=r-2k}^{r-k-2}j}{(k+1)!}\cdot c_{\mathbf{i};d}\cdot\frac{\prod\limits_{j=0}^k(i_j+j)}{(r+1)\prod\limits_{j=r-2k}^{r-k-1}j}
\end{eqnarray*}
\begin{align}
(k+2)!&=\sum\limits_{i_0>0}\frac{r^2+r-2k-2}{(r+1)(r-k-1)}c_{\mathbf{i};d}\prod\limits_{j=0}^k(i_j+j).\label{XX}
\end{align}
Consider $K_{r-2k-2,2k+2}$. For the pure diagrams characterised by $\mathbf{e}$ contributing to $K_{r-2k-2,2k+2}$, we have $e_{r-2k-2}=r-2k-2+2k+2=r$, which implies that $e_{r-2k-1}=r+1\Rightarrow i_0=2$. So only $\pi_k((2,\cdots,2);d)$ contributes to $K_{r-2k-2,2k+2}$, meaning that
$$\frac{r-2k-1}{\deg(\Sigma_k)}=c_{(2,\cdots,2);d}\dim K_{r-2k-2,2k+2}(\pi_k(2,\cdots,2);d).$$
By the definition,
$\dim K_{r-2k-2,2k+2}(\pi_k(2,\cdots,2);d)=\displaystyle\frac{(r-2k-1)!}{r\prod\limits_{e_p\not=r}|r-e_p|}$.\\
Since we have
\begin{eqnarray*}
\prod\limits_{e_p\not=r}|r-e_p|\cdot(r-1)\cdots(r-k-1)\prod\limits_{j=0}^k|j+1|&=&(r-1)!\\
\prod\limits_{e_p\not=r}|r-e_p|&=&\frac{(r-k-2)!}{(k+1)!}
\end{eqnarray*}
we know that $\dim K_{r-2k-2,2k+2}=\displaystyle\frac{(k+1)!}{r\prod\limits_{j=r-2k}^{r-k-2}j}.$\\
As a result,
\begin{eqnarray*}
c_{(2,\cdots,2);d}&=&\frac{r-2k-1}{\deg(\Sigma_k)\dim K_{r-2k-2,2k+2}}\\
&=&(r-2k-1)\cdot\frac{(k+1)!}{(r^2+r-2k-2)\prod\limits_{j=r-2k}^{r-k-2}j}\cdot\frac{r\prod\limits_{j=r-2k}^{r-k-2}j}{(k+1)!}\\
&=&\frac{r^2-2kr-r}{r^2+r-2k-2}.
\end{eqnarray*}
We now plug the value of $c_{(2,\cdots,2);d}$ in the equation \eqref{XX}:
\begin{eqnarray*}
(k+2)!&=&\sum\limits_{i_0>0}\frac{r^2+r-2k-2}{(r+1)(r-k-1)}c_{\mathbf{i};d}\prod\limits_{j=0}^k(i_j+j)\\
(k+2)!&=&\frac{r^2+r-2k-2}{(r+1)(r-k-1)}\frac{r^2-2kr-r}{r^2+r-2k-2}(k+2)!+\\
&&\sum\limits_{i_0=1}\frac{r^2+r-2k-2}{(r+1)(r-k-1)}c_{\mathbf{i};d}\prod\limits_{j=0}^k(i_j+j)
\end{eqnarray*}
Moving the term with $(k+2)!$ to the left hand side and canceling the factor $r^2+r-2k-2$, we get
\begin{align}
(k+2)![1-\frac{r^2-2kr-r}{(r+1)(r-k-1)}]=\sum\limits_{i_0=1}\frac{r^2+r-2k-2}{(r+1)(r-k-1)}c_{\mathbf{i};d}\prod\limits_{j=0}^k(i_j+j)\label{XXX}
\end{align}
For the right hand side, note that for each $\mathbf{i}$, we have $i_0=1$ and $i_j\leq2$ for $1\leq j\leq k$. Therefore $$\prod\limits_{j=0}^k(i_j+j)\leq(1+0)(2+1)\cdots(2+k)=\frac{(k+2)!}{2}.$$
Applying this inequality to \eqref{XXX}, we know that
\begin{eqnarray*}
(k+2)![1-\frac{r^2-2kr-r}{(r+1)(r-k-1)}]&\leq&\sum\limits_{i_0=1}\frac{r^2+r-2k-2}{(r+1)(r-k-1)}c_{\mathbf{i};d}\frac{(k+2)!}{2}\\
(k+2)!(rk+r-k-1)&\leq&\sum\limits_{i_0=1}(r^2+r-2k-2)c_{\mathbf{i};d}\frac{(k+2)!}{2}\\
&=&\frac{(k+2)!}{2}(r^2+r-2k-2)\sum\limits_{i_0=1}c_{\mathbf{i};d}\\
&\leq&\frac{(k+2)!}{2}(r^2+r-2k-2)(1-c_{(2,\cdots,2);d})\\
&=&\frac{(k+2)!}{2}(2kr+2r-2k-2).
\end{eqnarray*}
The inequalities are actually equalities! This implies, firstly, only the coefficient of $(1,2,\cdots,2)$ is non-zero among $\mathbf{i}$ with $i_0=1$, and secondly, $\sum\limits_{i_0=1}c_{\mathbf{i};d}=c_{(1,2,\cdots,2);d}$ is exactly $1-c_{(2,\cdots,2);d}$. In other words, the Betti diagram of $\Sigma_k$ is decomposed into the combination of 2 pure diagrams, which are represented by $(2,\cdots,2)$ and $(1,2,\cdots,2)$.

The degree sequence associated to $(2,\cdots,2)$ is $(0,k+2,\cdots,r-k-2,r,r+1)$ and that associated to $(1,2,\cdots,2)$ is $(0,k+2,\cdots,r-k-2,r-1,r+1)$. The pure diagrams contributing to $K_{r-2k-2,2k+1}$ must satisfy $e_{r-2k-2}=r-1$. So only $(1,2,\cdots,2)$ contributes to this term. Then we have
\begin{eqnarray*}
\dim K_{r-2k-2,2k+1}&=&\deg(\Sigma_k)\cdot c_{(1,2,\cdots,2)}\cdot\frac{(r-2k-1)!}{(r-1)(\prod\limits^{r-k-3}_{j=k+1}j)2}\\
&=&\frac{(r^2+r-2k-2)\prod\limits_{j=r-2k}^{r-k-2}j}{(k+1)!}\cdot\frac{2(k+1)(r-1)}{r^2+r-2k-2}\cdot\frac{(r-2k-1)!}{(r-1)(\prod\limits^{r-k-3}_{j=k+1}j)2}\\
&=&r-k-2.
\end{eqnarray*}
Moreover, both of these pure diagrams have linear strands of length $r-2k-3$. This finishes the proof.
\end{proof}
\begin{cor}With the same notations as Theorem 4.3, for $1\leq i\leq r-2k-3$, we have
\begin{eqnarray*}
&&\dim K_{i,k+1}(\Sigma_k,\mathcal{O}_{\Sigma_k}(1))\\
&=&\frac{(r-k-2)![r^3-(i+k+1)r^2-(i+k+2)r+2(k+1)(i+k+1)]}{(k+1)!(i+k+1)(i-1)!(r-i-2k-3)!(r-i-k-2)(r-i-k-1)(r-i-k)}.
\end{eqnarray*}
\end{cor}
\begin{proof}By the definition of pure diagrams, we have
$$\dim K_{i,k+1}(\pi_k((2,\cdots,2);d))=\frac{(r-2k-1)!}{(i+k+1)(p-1)!(r-i-2k-3)!(r-i-k-1)(r-i-k)},$$
$$\dim K_{i,k+1}(\pi_k((1,2,\cdots,2);d))=\frac{(r-2k-1)!}{(i+k+1)(i-1)!(r-i-2k-3)!(r-i-k-2)(r-i-k)}.$$
Plugging them in $\dim K_{i,k+1}$, we get
\begin{eqnarray*}
&&\dim K_{i,k+1}\\
&=&\deg(\Sigma_k)[c_{(2,\cdots,2);d}\dim K_{i,k+1}(\pi_k((2,\cdots,2);d))+c_{(1,2,\cdots,2);d}\dim K_{i,k+1}(\pi_k((1,2,\cdots,2);d))]\\
&=&\deg(\Sigma_k)[\frac{r^2-2kr-r}{r^2+r-2k-2}\frac{(r-2k-1)!}{(i+k+1)(i-1)!(r-i-2k-3)!(r-i-k-1)(r-i-k)}\\
&&+\frac{2(k+1)(r-1)}{r^2+r-2k-2}\frac{(r-2k-1)!}{(i+k+1)(i-1)!(r-i-2k-3)!(r-i-k-2)(r-i-k)}]\\
&=&\frac{\deg(\Sigma_k)(r-2k-1)![(r^2-2kr-r)(r-i-k-2)+2(k+1)(r-1)(r-i-k-1)]}{(r^2+r-2k-2)(i+k+1)(i-1)!(r-i-2k-3)!(r-i-k-2)(r-i-k-1)(r-i-k)}.
\end{eqnarray*}

Recall that $\deg(\Sigma_k)=\displaystyle\frac{(r^2+r-2k-2)\prod\limits_{j=r-2k}^{r-k-2}j}{(k+1)!}$. We have
$$\dim K_{i,k+1}=\frac{(r-k-2)![r^3-(i+k+1)r^2-(i+k+2)r+2(k+1)(i+k+1)]}{(k+1)!(i+k+1)(i-1)!(r-i-2k-3)!(r-i-k-2)(r-i-k-1)(r-i-k)}.$$
\end{proof}
For the case $k=0$, the secant variety is just the curve itself. The problem about the length of linear strand is the Green-Lazarsfeld Conjecture and was solved in \cite{EL15}. But I would like to calculate it again using Boij-S\"{o}derberg theory, which can give accurate values.
\begin{thm}Let $C\subset\mathbb{P}^r$ be a smooth projective curve of genus $2$ embedded via the linear system of a line bundle $L$ of degree at least $5$. Then we have:\\
(1)$\dim K_{r-2,1}(C,L)=r-2$;\\
(2)$\dim K_{i,1}(C,L)\not=0\iff 1\leq i\leq r-2$;\\
(3)For $1\leq i\leq r-2$, we have $\dim K_{i,1}(C,L)=\displaystyle\frac{(r-1)!(r^2-ir-2i-2)}{(i+1)(i-1)!(r-i)!}$.
\end{thm}
\begin{proof}
Similarly as above, only the sequences with $i_0>0$ contribute to $K_{r-1,2}$. The diagram associated to $i_0=1$ is characterised by $\{0,2,3,\cdots,r-1,r+1\}$ while that to $i_0=2$ is characterised by $\{0,2,3,\cdots,r-2,r,r+1\}$. The degree of the curve is $r+2$. Therefore we have
\begin{equation}
\frac{\dim K_{r-1,2}}{r+2}=\frac{2}{r+2}=c_{1;d}\cdot\frac{(r-1)!}{(r+1)(r-1)!}+c_{2;d}\cdot\frac{(r-1)!}{(r-3)(r-1)\cdots3}.\label{X2}
\end{equation}
On the other hand, notice that $\dim K_{r-2,2}=r-1$ and $\pi_0(2;d)$ is the only diagram contributing to $K_{r-2,2}$. By the definition, $\dim K_{r-2,2}(\pi_0(2;d))=\displaystyle\frac{(r-1)!}{r\cdot(r-2)!}$. Therefore
$$c_{2;d}\cdot\frac{(r-1)!}{r\cdot(r-2)!}=\frac{r-1}{r+2},$$
implying that $c_{2;d}=\displaystyle\frac{r}{r+2}$. Plugging the value of $c_{2;d}$ in \eqref{X2}, we find that $c_{1;d}=\displaystyle\frac{2}{r+2}$. Since $c_{1;d}$ and $c_{2;d}$ sum up to $1$, only these two corresponding pure diagrams contribute to the Betti diagram of the curve.Since only $\pi_0(1;d)$ contributes to $K_{r-2,1}$, we have
\begin{eqnarray*}
\frac{\dim K_{r-2,1}}{r+2}&=&c_{1;d}\dim K_{r-2,1}(\pi_0(1;d))
\end{eqnarray*}
\begin{eqnarray*}
\frac{\dim K_{r-2,1}}{r+2}&=&\frac{2}{r+2}\cdot\frac{(r-1)!}{(r-1)(r-3)!2}\\
\dim K_{r-2,1}&=&r-2.
\end{eqnarray*}
Since $\pi_0(1;d)$ has a linear strand of length $r-2$ while $\pi_0(2;d)$ has that of length $r-3$, their sum has a linear strand of length $r-2$.

For (3), similar to the calculation for $k\geq1$, we have
$$\dim K_{i,1}(\pi_0(1;d))=\frac{(r-1)!}{(i+1)(i-1)!(r-i-2)!(r-i)},$$
$$\dim K_{i,1}(\pi_0(2;d))=\frac{(r-1)!}{(i+1)(i-1)!(r-i-3)!(r-i-1)(r-i)}.$$
Plugging these into $\dim K_{i,1}$, we get
\begin{eqnarray*}
\dim K_{i,1}&=&\deg(C)[c_{1;d}\dim K_{i,1}(\pi_0(1;d))+c_{2;d}\dim K_{i,1}(\pi_0(2;d))]\\
&=&(r+2)[\frac{2}{r+2}\frac{(r-1)!}{(i+1)(i-1)!(r-i-2)!(r-i)}\\
&&+\frac{r}{r+2}\frac{(r-1)!}{(i+1)(i-1)!(r-i-3)!(r-i-1)(r-i)}]\\
&=&\frac{(r-1)!(r^2-ir-2i-2)}{(i+1)(i-1)!(r-i)!}.
\end{eqnarray*}
\end{proof}
At the end of this paper, I will prove a complemental result of $N_{k+2,r-2k-3}$. We can observe that $K_{i,2k+2}(\Sigma_k,\mathcal{O}_{\Sigma_k}(1))\not=0$ if and only if $r-2k-2\leq i\leq r-2k-1$. For curves of genus $2$, this can be easily seen from the shapes of pure diagrams contributing to the Betti diagrams of secant varieties, but I would prove a more general result for curves of general genus.
\begin{prop}Let $C$ be a curve of genus $g\geq1$. Let $L$ be a very ample line bundle with $\deg(L)\geq2g+2k+1$ and hence $H^0(L-K)\not=0$. Assume that the embedding $C\hookrightarrow\mathbb{P}^r$ is defined by the complete linear system of $L$. Let $\Sigma_k$ be the $k$-th order secant variety of $C$. Then
$$K_{i,2k+2}(\Sigma_k,\mathcal{O}_{\Sigma_k}(1))\not=0\iff r-g-2k\leq i\leq r-1-2k.$$
\end{prop}
\begin{proof}The implication from left to right was proved in Theorem 5.2 (4) of \cite{ENP}.\\
Now assume that $r-g-2k\leq i\leq r-1-2k$. We want to show $K_{i,2k+2}(\Sigma_k,\mathcal{O}_{\Sigma_k}(1))\not=0$. We essentially use the ideas in the proof of Proposition 5.1 of \cite{EL12}. Notice that we have the duality $K_{i,2k+2}(\Sigma_k,\mathcal{O}_{\Sigma_k}(1))\cong K_{r-(2k+1)-i,0}(\Sigma_k,\omega_{\Sigma_k};\mathcal{O}_{\Sigma_k}(1))^\vee$. Essentially we have to show the morphism
$$\bigwedge\limits^i H^0(\Sigma_k,\mathcal{O}_{\Sigma_k}(1))\otimes H^0(\Sigma_k,\omega_{\Sigma_k})\to\bigwedge\limits^{i-1} H^0(\Sigma_k,\mathcal{O}_{\Sigma_k}(1))\otimes H^0(\Sigma_k,\omega_{\Sigma_k}(1))$$
is not injective for $0\leq i\leq g-1$. We are interested in the case $i\geq1$. Take $f_0,\cdots,f_i$ linearly independent sections of $H^0(C,K_C)$ and $s\in H^0(C,L-K)$. Then $\forall F\in S^kH^0(C,K_C)$ nonzero, the element
$$\sum\limits_{j=0}^if_0s\wedge\cdots\wedge\widehat{f_js}\wedge\cdots\wedge f_is\otimes Ff_j$$ is mapped to $0$. Here we identify $H^0(\Sigma_k,\omega_{\Sigma_k})$ with $S^{k+1}H^0(C,K_C)$ and $H^0(\Sigma_k,\mathcal{O}_{\Sigma_k}(1))$ with $H^0(C,L)$.
\end{proof}

\bibliography{Syzygies_of_secant_varieties_of_curves_of_genus_2_v2}{}
\bibliographystyle{alpha}
\Addresses

\end{document}